\documentclass[12pt]{amsart}
\usepackage[T1]{fontenc}
\usepackage[utf8]{inputenc}
\usepackage{color,amssymb,amsmath,mathrsfs, enumerate,esint}
\usepackage{a4wide}
\usepackage{hyperref}
\usepackage{pgf,tikz,pgfplots}
\pgfplotsset{compat=1.15}
\usepackage{mathrsfs}
\usetikzlibrary{arrows}
\theoremstyle{plain}
\newtheorem{Theorem}{Theorem}[section]
\newtheorem{Lemma}[Theorem]{Lemma}

\newtheorem{Corollary}[Theorem]{Corollary}

\theoremstyle{definition}
\newtheorem{Remark}[Theorem]{Remark}
\newtheorem{Observation}[Theorem]{Observation}
\newtheorem{Definition}[Theorem]{Definition}

\DeclareMathOperator{\er}{\mathbb{R}}

\DeclareMathOperator*{\loc}{loc}
\DeclareMathOperator{\supp}{supp}

\title[G-N inequality for r. i. Banach function spaces]{Gagliardo--Nirenberg inequality via a new pointwise estimate}
\author{Karol Le\'snik}

\author{Tom\' a\v s Roskovec}

\author{Filip Soudský}

\address{K.~Le\'snik: Faculty of Mathematics and Computer Science,  Adam Mickiewicz University in Pozna{\'n},
  Uniwersytetu Pozna{\'n}skiego 4,
  61-614 Pozna{\'n}, Poland}
\email{klesnik@vp.pl}

\address{T.~G.~Roskovec: Faculty of Education, University of South Bohemia, Studentsk\' a 13, \v Cesk\' e Bud\v ejovice, Czech Republic}
\email{troskovec@jcu.cz}

\address{F.~Soudsk\' y: Faculty of Science, Humanities and Education, Technical University of Liberec, Studentská 1402/2, Liberec, Czech Republic}
\email{filip.soudsky@tul.cz}

\begin{document}

\begin{abstract}
We prove a new type of pointwise estimate of the Ka{\l}amajska--Mazya--Shapo\-shni\-kova type, where sparse averaging operators replace the maximal operator. It allows us to extend the Gagliardo--Nirenberg interpolation inequality to all rearrangement invariant Banach function spaces without any assumptions on their upper Boyd index, i.e. omitting problems caused by unboundedness of maximal operator on spaces close to $L^1$. In particular, we remove unnecessary assumptions from the Gagliardo--Nirenberg inequality in the setting of Orlicz and Lorentz spaces. 
%maximal operator not being strong type on spaces close to $L^1$.
The applied method is new in this context and may be seen as a kind of sparse domination technique fitted to the context of rearrangement invariant Banach function spaces. 
\end{abstract}

\maketitle
\footnotetext[0]{
2020 \textit{Mathematics Subject Classification}: 46E35, 35A23, 26D10}
\footnotetext[0]{\textit{Key words and phrases}:  Gagliardo--Nirenberg inequality, rearrangement invariant Banach function spaces,  Calder\'on--Lozanovskii spaces, sparse domination}

\section{Introduction}

There are a few forms of inequalities called (Sobolev--)Gagliardo--Nirenberg inequalities, which have been investigated in plenty of contexts and remain at the center of interest up to current days. Some involve only one derivative (see \cite{DDN23}, \cite{Lady}, \cite{MRR}, \cite{N}) and are more like Sobolev--type inequalities, while others deal with two derivatives of different orders and a function itself, like in \cite{BS}, \cite{BrMi},  \cite{DuDS23}, \cite{KC09} and \cite{WeiWY23}. We are interested in the latter kind, i.e. in the inequality of the type 
\begin{equation}\label{introductionGNri}
\big\|\nabla^j u\big\|_{Z}\leq C\big\|\nabla^k u\big\|^{\frac{j}{k}}_X\|u\|^{1-\frac{j}{k}}_Y,
\end{equation}
where, traditionally, by the expression $\|\nabla^k u\|_X$ we mean
\begin{equation}\label{nabla}
\|\nabla^{k} u\|_X:=\Big\|\sum_{|\alpha|=k}\Bigl|\frac{\partial^\alpha u}{\partial x^\alpha}\Bigr|\Big\|_X.
\end{equation}

Such inequality has been continuously investigated in many theoretical and practical contexts for years. More precisely, from the theoretical point of view, it is considered not only for different classes of spaces $X, Y,$ and $Z$, as Orlicz spaces (cf. \cite{KalPP} and \cite{GNOrlicz-Formica}), Lorentz spaces (cf. \cite{WeiWY23}), variable exponent spaces (cf. \cite{KC09}) or weighted Lebesgue spaces (cf. \cite{DuDS23}) but also for fractional derivatives (cf. \cite{BrMi}).

We are interested in \eqref{introductionGNri} for rearrangement invariant (r.i. for short) Banach function spaces.
Thus, we continue our investigations initiated in  \cite{LeRoSo23, FiFoRoSo2}. 
In \cite{LeRoSo23}, we proved that \eqref{introductionGNri} holds with $Z=X^{j/k}Y^{1-j/k}$ being the Calder\'on-Lozanovskii space and for natural numbers $0<j<k$, but under assumption that maximal operator is bounded on spaces $X, Y$ (see also \cite{FiFoRoSo2} for another, not necessarily equivalent, formulation). As is proved in \cite{LeRoSo23}, for given $X$ and $Y$, such choice of $Z$ is already optimal among rearrangement invariant space in almost all crucial cases for \eqref{introductionGNri} to hold. In the present paper, our goal is to remove an assumption on the boundedness of the maximal operator in spaces $X, Y$.  

The main tool to prove \eqref{introductionGNri} in both  papers mentioned above was Kałamajska--Mazya--Shaposhnikova inequality (see \cite{Kal, MS}), which states that
\begin{equation}\label{MS}
|\nabla^j u|\leq C (M\nabla^k u)^{\frac{j}{k}}(Mu)^{1-\frac{j}{k}},\ \ {\rm\ where\ }1\leq j < k,
\end{equation}
where $M$ is the classical maximal operator, and $C$ is a constant depending only on dimension.

Notice that \eqref{MS} applies quite directly to the context of r.i. Banach function spaces (also to each Banach function lattice, cf. \cite{KC09}), but we cannot get the full range of the r.i. Banach function spaces by this approach. In fact, it does not cover spaces close to $L^1$ (precisely, spaces with the upper Boyd index equal to $1$) due to the fact that the maximal operator is unbounded therein. On the other hand, it is known from the classical Gagliardo--Nirenberg inequality for Lebesgue spaces (\cite{Ni}, cf. \cite{FIFOROSO}) that \eqref{GNri} holds for $X=Y=Z=L^1$. It indicates that \eqref{GNri} itself does not reveal symptoms of weak type, and the restrictions on upper Boyd indices from  \cite{FiFoRoSo2, LeRoSo23} are not necessary but come instead from the method applied in the proof. At the same time, neither Gagliardo's nor Nirenberg's methods apply beyond Lebesgue spaces. Also, tools of interpolation theory are not efficient for spaces close to $L^1$. Therefore, to prove  \eqref{GNri} for all r.i. Banach function spaces spaces, we need a new approach. 

In the paper, we achieve the described goal and remove the mentioned assumptions. Namely, we prove the following: 

\begin{Theorem}[Gagliardo-Nirenberg interpolation inequality for r.i. Banach function spaces]\label{MT}
Let $X, Y$ be r.i. Banach function spaces over $\er^n$ and let $Z:=X^{j/k}Y^{1-j/k}$ be  the Calder\'on--Lozanovskii space, where $1\leq j< k$. Then, there exists a positive constant $C_a$ (independent of  dimension) such that for each $1\leq i\leq n$ 
inequality holds
\begin{equation}\label{pureGNri}
\left\|\frac{\partial^j u}{\partial x^j_i}\right\|_{Z}\leq C_a\left\|\frac{\partial^k u}{\partial x^k_i}\\\right\|^{\frac{j}{k}}_X\|u\|^{1-\frac{j}{k}}_Y
\end{equation}
for each $u\in W^{k,1}_{\loc}(\er^n)$. In particular, for each $u\in W^{k,1}_{\loc}(\er^n)$, there holds 
\begin{equation}\label{GNri}
\|\nabla^j u\|_{Z}\leq C_b\|\nabla^k u\|^{\frac{j}{k}}_X\|u\|^{1-\frac{j}{k}}_Y,
\end{equation}
where positive constant $C_b$ depends only on dimension $n$.
\end{Theorem}

Therefore, in light of the optimality of the choice of $Z$ proved \cite{LeRoSo23}, it practically closes the discussion on the Gagliardo--Nirenberg inequality \eqref{introductionGNri} for r.i. Banach function spaces. In particular, it covers and strengthens analogous inequalities known previously for Orlicz or Lorentz spaces (\cite{KalPP, FiFoRoSo2, LeRoSo23, WeiWY23}). 

Notice also that, unlike as for Lebesgue spaces, in the case of general r.i. Banach function spaces, it is not evident if  \eqref{pureGNri} and \eqref{GNri} are equivalent. Thus, we believe it is worth stating \eqref{pureGNri} separately, although classically, the Gagliardo--Nirenberg inequality is rather understood as \eqref{GNri}  Detailed discussion on differences between \eqref{pureGNri} and \eqref{GNri} is postponed till Section \ref{questsection}. 

Anyhow, perhaps more interesting than the result itself is the method applied and its consequences, which, as a byproduct, gives new types of pointwise estimates for partial derivatives.

The method we employ to prove Theorem \ref{MT} may be regarded as a kind of sparse domination technique, which after Lerner's paper \cite{Le13a} appeared to be a very fruitful tool in Calder\'on--Zygmund theory and, in general, in the classical harmonic analysis \cite{Le13b, Le16, Lec17, LecMA17, LeOm20, LeLo22, DuDS23}.
Roughly speaking, the sparse domination method involves the replacement of the classical maximal operator by averaging operators (also called sparse operators), which, however, have to be chosen individually for each element. Once we can keep a uniform bound of all such operators, then the job is done. 

Let us briefly explain how it works in our context. Given a countable family of finite measure sets $\mathcal{P}$ (in principle not disjoint)  we define  formally on $L^1_{\loc}$ the operator $T_{\mathcal{P}}$ by the formula 
\begin{equation}\label{Oper}
T_{\mathcal{P}}u=\sum_{P\in \mathcal{P}}\fint_P u(t)dt  \chi_P,
\end{equation}
where $\chi_P$ stands for the characteristic function of $P$ and 
\[
\fint_P u(t)dt:=\frac{1}{|P|}\int_P u(t)dt.
\]
If additionally, we assume that 
\[
\sum_{P\in \mathcal{P}} \chi_P\leq K\chi_{\er^n},
\]
then it is easy to see that such $T_{\mathcal{P}}$ is bounded on $L^1$ and on $L^{\infty}$ with the norm  less or equal to $K$ in both cases, thus also bounded by $K$ in each r.i. Banach function space (cf. Lemma \ref{lemCM}).

We focus on the case $k=2, j=1$ since other cases can be reached by induction.
The idea is to replace the maximal operator $M$ from \eqref{MS}, which is unbounded on $L^1$, by the operators $T_{\mathcal{P}}$, which in turn are bounded on each r.i. Banach function space, to get the pointwise inequality with universal constant $C$ 
\begin{equation}\label{SD1}
|\nabla u|^2\leq C T_{\mathcal{P}}|\nabla^2 u| T_{\mathcal{P}}|u|.
\end{equation}
However, such an operator $T_{\mathcal{P}}$ cannot be chosen universally for all $u$'s, as is the maximal operator in \eqref{MS}, but the family $\mathcal{P}$ defining $T_{\mathcal{P}}$ has to depend on the function $u$. Consequently, we need to show that for each  $u\in \mathcal{C}_c^2(\er^n)$ there is a family $\mathcal{P}:=\mathcal{P}(u)$ such that 
 \eqref{SD1} holds with some universal constant independent of $u$ and $\mathcal{P}$. It still wouldn't be enough to get \eqref{GNri}, unless we know there  is another universal constant $K$, depending at most on the dimension $n$, such that all families  $\mathcal{P}$ satisfy 
$$
\sum_{P\in \mathcal{P}} \chi_P\leq K\chi_{\er^n}.
$$
Concluding, the essence of the method is contained in the following crucial theorem.

\begin{Theorem}\label{crux}
There is positive constant $C$ depending only on dimension $n$ such that for each $u\in \mathcal{C}^2_c(\er^n)$ and each $i\in\{1,\dots,n\}$ there exists a countable family $\mathcal{P}$ of measurable subsets of $\mathbb{R}^n$ satisfying
\begin{equation}\label{BodOdhad}
    \left|\frac{\partial u}{\partial x_i}\right|^2\leq C
    \sum_{P\in\mathcal{P}} \left(\fint_P \left|\frac{\partial^2 u}{\partial x_i^2}(t) \right|dt\fint_P |u(t)|dt \right)\chi_P 
\end{equation}
and 
    $$
    \sum_{P\in\mathcal{P}}\chi_P\leq 5 \chi_{\er^n}.
    $$
\end{Theorem}

In particular, from Theorem \ref{crux} it follows.

\begin{Corollary}\label{corcrux}
There are two positive constants $C$ and $K$ depending only on dimension $n$ such that for each $u\in \mathcal{C}^2_c(\er^n)$  there exists a countable family $\mathcal{P}$ of measurable subsets of $\mathbb{R}^n$ such that 
$$
    \left|\nabla u\right|^2\leq C
    T_{\mathcal{P}}|\nabla^2 u|T_{\mathcal{P}}| u| 
$$ 
and 
$$
    \sum_{P\in\mathcal{P}}\chi_P\leq K\chi_{\er^n}.
    $$
\end{Corollary}

From these, \eqref{GNri} and \eqref{pureGNri} follow immediately for compactly supported functions, which allows us to conclude Theorem \ref{MT} for separable r.i. Banach function spaces. Notice, however, that to prove it in full generality, we need a slightly more technical analogue of Theorem \ref{crux}, namely Lemma \ref{Almpoint} that will be formulated in Section 2.

The paper is organized as follows. Section \ref{sparsesection} contains the announced crucial pointwise estimates. Section \ref{GNsection} starts with some background on function spaces theory and contains the proof of Gagliardo--Nirenberg inequality for general r.i. Banach function spaces,  while in Section \ref{orliczsection}, we apply it to Orlicz and Lorentz spaces. Finally, the last  Section \ref{questsection} is devoted to a discussion of the advantages and consequences of the method. Note that throughout the paper, we use $K$ and $C$ for universal constants whose value and dependence differ from theorem to theorem; in proofs, the value of $C$ may differ in each step.

\section{Sparse domination}\label{sparsesection}

 We start with the one-dimensional case, which is, on the one hand, a particular case of Theorem \ref{crux} (regardless of assumption on the compact support) and Lemma \ref{Almpoint}, being, on the other hand, the crux to the multidimensional case. 
 
\begin{Lemma}\label{sparse1dim}
There exists a positive constant $C$, such that for every $u\in \mathcal{C}^2\cap (L^1+L^{\infty})(\er)$ there exists a countable family  $\mathcal{P}$ of bounded open intervals satisfying 
\begin{equation}\label{KLIC}
|u'|^2\leq C \sum_{P\in \mathcal{P}} \left(\fint_P|u''(s)|ds \fint_P|u(s)|ds\right)\chi_P
\end{equation}
and 
$$\sum_{P\in\mathcal{P}}\chi_P\leq 3\chi_{\er}.$$
\end{Lemma}

\begin{proof}
Evidently, the inequality \eqref{KLIC} needs to be verified only for $x\in\{u'\neq 0\}$. 
We will design the family $\mathcal{P}_+$ covering the set $\{u'>0\}$, and then we will prove its desired properties. Once they are verified, the proof is simply finished by the choice $\mathcal{P}=\mathcal{P}_+\cup \mathcal{P}_-$, where $\mathcal{P}_-$ is defined analogously to cover the set $\{u'<0\}$.

Define exhausting pairwise disjoint covering of $\{u'>0\}$
$$
E_k:=\{2^{k-1}\leq u'< 2^{k}\},
$$
where $k\in \mathbb{Z}$. 
Suppose that $x\in E_k$ for some $k$. Let us denote 
\begin{equation}\label{Intervaly}
\begin{aligned}
y:=y(x)=&\sup\{\tau\in \er:(x,\tau)\subset E_{k-1}\cup E_k\cup E_{k+1}\},\\
z:=z(x)=&\inf\{\tau\in \er:(\tau,x)\subset E_{k-1}\cup E_k\cup E_{k+1}\}.
\end{aligned}
\end{equation}

Notice that for arbitrary $x\in\{u'>0\}$ interval $(z(x),y(x))$ is nonempty and bounded. In fact, if $x\in E_k$, then $x$ is in the interior of $E_{k-1}\cup E_k$, thus $(z(x),y(x))$ is nonempty. For $s\in (x,y(x))$ one has
$$
u(s)=u(x)+\int_x^s u'(\tau)d\tau\geq u(x)+(s-x)\frac{u'(x)}{4}.
$$
Thus, if $y(x)$ was infinite, the function $u$ would not be in the space $L^1+L^\infty$. It means that $y(x)<\infty$. Similarly, $z(x)>-\infty$.

Define
$$
\mathcal{P}_+:=\left\{(y(x),z(x))\right\}_{x\in \{u'>0\}}.
$$
Now, we prove the bound on the overlapping 
\begin{equation}\label{CLICK3+}
\sum_{I\in\mathcal{P}_+}\chi_I\leq 3\chi_{\er}.
\end{equation}
In fact, for each integer $k$ and given two points $s,t\in E_k$ one of two possibilities holds;
$$
\textup{either }(z(s),y(s))=(z(t),y(t))\quad \textup{or}\quad  (z(s),y(s))\cap (z(t),y(t))=\emptyset.
$$
In particular, $\mathcal{P}_+$ is countable. 
Furthermore, if $t\in E_k$, then it may happen that $t\in(z(s),y(s))$ only for $s\in E_{k-1}\cup E_k\cup E_{k+1}$, which means that \eqref{CLICK3+} holds.
It remains to prove \eqref{KLIC} on $\{u'>0\}$.

%Finally, we establish some relations between $u, u'$ and $u''$ to prove the inequality \eqref{KLIC}.
Firstly, we will explain that 
\begin{equation}\label{Odhad1+}
    \frac{u'(x)}{4} \leq (y-z)\fint_z^y |u''(s)|ds,
\end{equation}
when $x\in(z,y):=(z(x),y(x))$.
%Based on the closeness of indices of $E_k$ containing $x, y$ and $z$, we estimate
%$$
% \min\left\{\int_x^y |u''(s)|ds,\int_z^x|u''(s)|ds\right\}\geq \frac{u'(x)}{4}.   
%$$
In fact, let $x\in E_k$, then $2^{k-1}\leq u'(x)< 2^{k}$, while either $u'(y)=2^{k-2}$ or $u'(y)=2^{k+1}$. In either case, we estimate
\[
\int_z^y |u''(s)|ds \geq \int_x^y |u''(s)|ds \geq \left|\int_x^y u''(s)ds\right| =|u'(y)-u'(x)|\geq \frac{1}{4} u'(x).
\]
This implies \eqref{Odhad1+}.

Secondly, we claim that
\begin{equation}\label{Odhad2+}
\frac{u'(x)}{32}\leq \frac{\fint_y^z|u(s)|ds}{y-z}.
\end{equation}
Since $u\in\mathcal{C}^2([z,y])$, there exists some $c\in(z,y)$ satisfying
$$
u(c)=\fint_z^y u(\tau)d\tau.
$$ 
By the standard argument (cf. for example \cite[Lemma 3.1]{FIFOROSO}) we get
%We estimate by \cite[Lemma 3.1]{FIFOROSO} in case $p=1$ following
$$
\begin{aligned}
    \int_z^y |u(s)|ds&\geq \frac{1}{2}\int_z^y|u(s)-u(c)|ds=\frac{1}{2}\int_z^y\left|\int_c^s u'(\tau)d\tau\right|ds.
\end{aligned}
$$
However,  $(z,y)\subset  E_{k-1}\cup E_k\cup E_{k+1}$, thus for each $z<s<y$
\[
\left|\int_c^s u'(\tau)d\tau\right|\geq 2^{k-2}|s-c|\geq \frac{1}{4}u'(x)|s-c|.
\]
Consequently, we get
$$
\begin{aligned}
     \int_z^y |u(s)|ds&{\geq} \frac{u'(x)}{8}\int_z^y|s-c|ds\geq \frac{u'(x)(y-z)^2}{32}.
\end{aligned}
$$
Furthermore, that implies \eqref{Odhad2+}.
Concluding, by estimations \eqref{Odhad1+}, \eqref{Odhad2+}, for each $x\in \{u'>0\}$ there holds
$$
 u'(x)^2\leq 128 \fint_{z(x)}^{y(x)}|u''(s)|ds\fint_{z(x)}^{y(x)}|u(s)|ds.   
$$
In the final thoughts, we remind that union of $\mathcal{P}=\mathcal{P}_+\cup\mathcal{P}_-$ does not affect desired property \eqref{KLIC} nor the bound on the overlap within the family.
\end{proof}

Let us point out some observations on the above proof that will be used to prove a more complicated multidimensional case.
\begin{Observation}\label{after}
  In the proof above we have shown precisely that for $$x\in \{2^{k-1}\leq u'<2^k\}$$ there holds 
  \begin{itemize}
      \item[(a)]
\[
u'(x)\leq 4\int_{y(x)}^{z(x)}|u''(s)| ds,
\]
      \item[(b)]
\[
u'(x)\leq \frac{32}{(y(x)-z(x))^2}\int_{y(x)}^{z(x)}|u(s)| ds.
\]
\end{itemize}
Observe that if $I$ is some interval such that $(z(x),y(x))\subset I$ and $I\subset E_{k-d}\cup...\cup E_{k+d}$, $d>1$, then there holds as well
\begin{itemize}
      \item[(a')]
\[
u'(x)\leq 4\int_{I}|u''(s)| ds,
\]
\item[(b')]
\[
u'(x)\leq 2^d\frac{32}{|I|^2}\int_I|u(s)| ds.
\]
\end{itemize}
Point (a') is evident, while (b') follows from the observation that only the lower estimate of $u'$ was used to get (b). 
\end{Observation}

Now, we can proceed with the multidimensional case.

\begin{proof}[Proof of Theorem \ref{crux}]
Let $u\in \mathcal{C}_c^2(\er^n)$. We agree on the convention that for $x\in \er^n$, $\bar x=(x_2,x_3,...,x_n)\in \er^{n-1}$ and $x= (x_1,\bar x)$. Without loss of generality, we will prove only the case of $i=1$. Exactly as in the one-dimensional case, we find required covering only for $\{\frac{\partial u}{\partial x_1}>0\}$.

For $k\in \mathbb{Z}$ define
\[
E_k=\left\{2^{k-1}\leq \frac{\partial u}{\partial x_1}<2^k\right\}.
\]
Let $k\in \mathbb{Z}$ and suppose that $(t,\bar x)\in E_k$. We will consider a function of one variable $u(\cdot,\bar x)$, reducing many explanations to the one-dimensional situation. As in the proof of Lemma \ref{sparse1dim} we define 
$$
\begin{aligned}
y_{\bar x}(t)=&\sup\left\{\tau\in \er:(t,\tau)\subset  \left\{2^{k-2}\leq \frac{\partial u}{\partial x_1}(\cdot,\bar x)<2^{k+1}\right\}\right\},\\
z_{\bar x}(t)=&\inf\left\{\tau\in \er:(\tau,t)\subset \left\{2^{k-2}\leq \frac{\partial u}{\partial x_1}(\cdot,\bar x)<2^{k+1}\right\}\right\}.
\end{aligned}
$$

For arbitrary $k\in\mathbb{Z}$, thanks to the uniform continuity of $u$ and its derivatives up to the second order, there is $\delta_k>0$ such that $|x-y|\leq \delta_k$ implies 
 \begin{equation}\label{modul}
   \max\left\{ \left|u(x)-u(y)\right|, \left|\frac{\partial u}{\partial x_1}(x)-\frac{\partial u}{\partial x_1}(y)\right|,\left|\frac{\partial^2 u}{\partial x_1^2}(x)-\frac{\partial^2 u}{\partial x_1^2}(y)\right| \right\}\leq \min\left\{2^{k-4}, \frac{2^{2k-4}}{M}\right\},
\end{equation}
where 
$$
M=\max\left\{\left\|u\right\|_\infty,\left\|\frac{\partial u}{\partial x_1}\right\|_{\infty},\left\|\frac{\partial^2 u}{\partial x^2_1}\right\|_{\infty}\right\}.
$$ 
The choice of $\delta_k$ will be justified by its multiple applications on estimates in the following part.

We define 
 \begin{equation}\label{defRk}
R_k=\bigcup_{(t,\bar x)\in E_k}(z_{\bar x}(t),y_{\bar x}(t))\times B(\bar x,\delta_k),
\end{equation}
where $B(\bar x,\delta_k)$ means the ball in $\er^{n-1}$. We will prove that $\{R_k\}_{k\in \mathbb{Z}}$ satisfies requirements of the theorem over $\{\frac{\partial u}{\partial x_1}>0\}$. First of all, notice that for each $k\in \mathbb{Z}$ it holds
\begin{equation}\label{inc}
R_k\subset E_{k-2}\cup E_{k-1}\cup E_{k}\cup E_{k+1}\cup E_{k+2}.
\end{equation}
Indeed, for arbitrary $y=(y_1,\bar y)\in R_k$ there exists $x=(y_1,\bar{x})\in E_k$ such that $y\in(z_{\bar{x}}(y_1),y_{\bar{x}}(y_1))\times B(\bar x, \delta_k)$. In particular, $|x-y|<\delta_k$ and hence, by the choice of $\delta_k$, we conclude
%Indeed, just by the choice of $\delta_k$ for arbitrary $y\in R_k$ there exists $x=(y_1,\bar{x})$ such that $y_1\in(z_{\bar{x}},y_{\bar{x}})$ note that $|x-y|<\delta_k$ hence by \eqref{modul} we have
$$
\left|\frac{\partial u}{\partial x_1}(y)\right|\geq \left|\frac{\partial u}{\partial x_1}(x)\right|-\left|\frac{\partial u}{\partial x_1}(x)-\frac{\partial u}{\partial x_1}(y)\right|\geq 2^{k-2}-2^{k-4}\geq 2^{k-3}, 
$$
similarly, it holds
$$
\left|\frac{\partial u}{\partial x_1}(y)\right|\leq  \left|\frac{\partial u}{\partial x_1}(x)\right|+\left|\frac{\partial u}{\partial x_1}(x)-\frac{\partial u}{\partial x_1}(y)\right|< 2^{k+1}+2^{k-4}<2^{k+2}.
$$
In particular, inclusion \eqref{inc} implies that 
\begin{equation}\label{remmm}
\sum_{k\in \mathbb{Z}}\chi_{R_k}\leq 5.
\end{equation}
It remains to show that for each $k\in \mathbb{Z}$  and each $(t,\bar x)\in E_k$ (actually, also for each  $(t,\bar x)\in R_k$ but with doubled constant) there holds 
$$
\left(\frac{\partial u}{\partial x_1}(t,\bar x)\right)^2\leq 
C \fint_{R_k}\left|\frac{\partial^2 u}{\partial x^2_1}(x)\right|dx \fint_{R_k}|u(x)|dx,
$$
with some constant $C$ independent of $u$ and of $k$. As we control the size of the left-hand side (over $E_k$) up to constant, it is enough to prove the estimate with $2^{2k}$ replacing the left-hand side.

We define
\[
\overline{R}_k:=\{\bar x\in \er^{n-1}: (t,\bar x)\in R_k {\rm \ for\ some\ } t\in \er \},
\]
\[
\widehat{R_k}(\bar x):=\{t\in \er: (t,\bar x)\in R_k\}.
\]
Notice firstly that each $R_k$ is open and bounded set, thus for each $\bar x\in \overline{R}_k$, the set $\widehat{R_k}(\bar x)$ is open and bounded subset of $\er$. It means, for each $\bar x\in \overline{R}_k$ there is at most countable family of disjoint open intervals $(I^k_j(\bar x))$ such that 
\[
\widehat{R_k}(\bar x)=\bigcup_jI^k_j(\bar x). 
\]
With this notation, by the Fubini theorem and the Cauchy--Schwartz inequality applied twice, we have the following estimates
$$
\begin{aligned}
&\left(\fint_{R_k}\left|\frac{\partial^2 u}{\partial x^2_1}(x)\right|dx \fint_{R_k}|u(x)|dx\right)^{\frac{1}{2}}\\
&=
\frac{1}{|R_k|}\left(\int_{\overline{R}_k}\int_{\widehat{R_k}(\bar x)}\left|\frac{\partial^2 u}{\partial x^2_1}(t,\bar x)\right|dt d\bar x  \int_{\overline{R}_k}\int_{\widehat{R_k}(\bar x)}|u(t,\bar x)|dt d\bar x \right)^{\frac{1}{2}}
\\
&\geq \frac{1}{|R_k|}\int_{\overline{R}_k} \left(\int_{\widehat{R_k}(\bar x)}\left|\frac{\partial^2 u}{\partial x^2_1}(t,\bar x)\right|dt \int_{\widehat{R_k}(\bar x)}|u(t,\bar x)|dt \right)^{\frac{1}{2}}d\bar x 
\\ 
&= \frac{1}{|R_k|}\int_{\overline{R}_k} \left(\sum_j\int_{I^k_j(\bar x)}\left|\frac{\partial^2 u}{\partial x^2_1}(t,\bar x)\right|dt  \sum_j\int_{I^k_j(\bar x)}|u(t,\bar x)|dt \right)^{\frac{1}{2}}d\bar x 
\\
&\geq 
\frac{1}{|R_k|}\int_{\overline{R}_k}\sum_j \left(\int_{I^k_j(\bar x)}\left|\frac{\partial^2 u}{\partial x^2_1}(t,\bar x)\right|dt  \int_{I^k_j(\bar x)}|u(t,\bar x)|dt \right)^{\frac{1}{2}}d\bar x 
\\
&=
\frac{1}{|R_k|}\int_{\overline{R}_k}\sum_j |I^k_j(\bar x)|\left(\fint_{I^k_j(\bar x)}\left|\frac{\partial^2 u}{\partial x^2_1}(t,\bar x)\right|dt  \fint_{I^k_j(\bar x)}|u(t,\bar x)|dt \right)^{\frac{1}{2}}d\bar x 
\end{aligned}
$$
Since 
\[
\frac{1}{|R_k|}\int_{\overline{R}_k} \sum_j|I^k_j(\bar x)|d\bar x
= \frac{1}{|R_k|}\int_{\overline{R}_k} |\widehat{R_k}(\bar x)|d\bar x=1,
\]
we need only to show that 
\[
2^{2k} \leq C\fint_{I^k_j(\bar x)}\left|\frac{\partial^2 u}{\partial x^2_1}(t,\bar x)\right|dt  \fint_{I^k_j(\bar x)}|u(t,\bar x)|dt
\]
for each $\bar x\in \bar{R}_k$ and each $j$.
Observe, however, that by definition of $R_k$, for each $I_j^k(\bar x)$ there is some $(t,\bar y)\in E_k$ such that 
\[
(z_{\bar y}(t),y_{\bar y}(t))\subset I_j^k(\bar x).
\]
Moreover,  $|\bar x - \bar y|<\delta_k$, choice of $\delta_k$ \eqref{modul} and inclusion $I_j^k(\bar x)\times \{\bar x\} \subset E_{k-2}\cup... \cup E_{k+2}$ based on \eqref{inc} imply that $I_j^k(\bar x)\times \{\bar y\} \subset E_{k-3}\cup... \cup E_{k+3}$.

It follows now from Observation \ref{after} that
\begin{itemize}
\item[(a)]
\[
2^{k-1}\leq \left|\frac{\partial u}{\partial x_1}(t,\bar y)\right|\leq 4 \int_{I_j^k(\bar x)}|u(s,\bar y)|ds
\]
\item[(b)]
\[
2^{k-1}\leq \left|\frac{\partial u}{\partial x_1}(t,\bar y)\right|\leq \frac{128}{|I_j^k(\bar x)|^2}\int_{I_j^k(\bar x)}\left|\frac{\partial^2 u}{\partial x^2_1}(s,\bar y)\right|ds
\]
\end{itemize}

To finish the proof, it is enough to notice that it holds
\[
\fint_{I^k_j(\bar x)}\left|\frac{\partial^2 u}{\partial x^2_1}(t,\bar y)\right|dt  \fint_{I^k_j(\bar x)}|u(t,\bar y)|dt\leq 2
\fint_{I^k_j(\bar x)}\left|\frac{\partial^2 u}{\partial x^2_1}(t,\bar x)\right|dt  \fint_{I^k_j(\bar x)}|u(t,\bar x)|dt.
\]
This is, however, guaranteed by the choice of $\delta_k$, which implies that 
\[
\begin{aligned}
&\fint_{I^k_j(\bar x)}\left|\frac{\partial^2 u}{\partial x^2_1}(t,\bar y)\right|dt  \fint_{I^k_j(\bar x)}|u(t,\bar y)|dt\\%%%%%%
&\leq 
\fint_{I^k_j(\bar x)}\left|\frac{\partial^2 u}{\partial x^2_1}(t,\bar y)-\frac{\partial^2 u}{\partial x^2_1}(t,\bar x)\right|dt  \fint_{I^k_j(\bar x)}|u(t,\bar x)|dt\\
&+
\fint_{I^k_j(\bar x)}\left|\frac{\partial^2 u}{\partial x^2_1}(t,\bar x)\right|dt  \fint_{I^k_j(\bar x)}|u(t,\bar y)-u(t,\bar x)|dt\\%%%%%%
&+\fint_{I^k_j(\bar x)}\left|\frac{\partial^2 u}{\partial x^2_1}(t,\bar y)-\frac{\partial^2 u}{\partial x^2_1}(t,\bar x)\right|dt  \fint_{I^k_j(\bar x)}|u(t,\bar y)-u(t,\bar x)|dt\\
&+\fint_{I^k_j(\bar x)}\left|\frac{\partial^2 u}{\partial x^2_1}(t,\bar x)\right|dt  \fint_{I^k_j(\bar x)}|u(t,\bar x)|dt\\
&\leq 3\cdot 2^{2k-4}+\fint_{I^k_j(\bar x)}\left|\frac{\partial^2 u}{\partial x^2_1}(t,\bar x)\right|dt  \fint_{I^k_j(\bar x)}|u(t,\bar x)|dt.
\end{aligned}
\]
Of course, the family $(R_k)$ covers only $\{\frac{\partial u}{\partial x_1}>0\}$, but analogously we may construct $(R'_k)$ that will cover $\{\frac{\partial u}{\partial x_1}<0\}$, so the proof is finished.
\end{proof}

%%%%%%%%%%%%%%%%%%%%%%%%%%%%%%%%%%%%%%%%%%%%%%5

\begin{proof}[Proof of Corollary \ref{corcrux}]
Let $u\in \mathcal{C}_c^2(\er^n)$. For each $1\leq i\leq n$ there is family $\mathcal{P}_i$ satisfying \eqref{BodOdhad} and
$$
\sum_{P\in\mathcal{P}_i}\chi_P\leq 5 \chi_{\er^n}.
$$
We define $\mathcal{P}=\bigcup_{i=1}^n \mathcal{P}_i$. Then, we extend the overlap estimate to 
$$
\sum_{P\in\mathcal{P}}\chi_P\leq 5n \chi_{\er^n}.
$$
Moreover, estimate
\[
\begin{aligned}
    \left|\nabla u\right|^2 &\leq n \sum_{i=1}^n\left|\frac{\partial u}{\partial x_i} \right|^2\\
    &\leq
     Cn\sum_{i=1}^n\sum_{P\in\mathcal{P}_i} \left(\fint_P \left|\frac{\partial^2 u}{\partial x_i^2}(t) \right|dt\fint_P |u(t)|dt \right)\chi_P \\
    &= Cn\sum_{P\in\mathcal{P}} \left(\fint_P \left|\frac{\partial^2 u}{\partial x_i^2}(t) \right|dt\fint_P |u(t)|dt \right)\chi_P \\
    &\leq Cn T_{\mathcal{P}} \left|\nabla^2u \right| T_{\mathcal{P}} |u|,
\end{aligned}
\]
thus, the statement holds.
\end{proof}

%%%%%%%%%%%%%%%%%%%%%%%%%%%%%%%%%%%%%%%%%%%%%%%%%%%%%%%%%%%5
Let us refine the above theorem by removing the assumption on compact support. While the idea essentially remains the same, we lose the elegance of the statement of Theorem \ref{crux}.

\begin{Lemma}\label{Almpoint}
There exists constant  $C>0$, depending only on dimension $n$, such that for each  $u\in\mathcal{C}^2\cap(L^1+L^\infty)(\er^n)$ and each $1\leq i\leq n$ there exists a sequence of open sets $(G_l)$, satisfying 
$\chi_{G_l}\to \chi_{\textup{supp}\frac{\partial u}{\partial x_i}}$ a.e.,  
and such that for each $l$ there exists a countable family of measurable sets $\mathcal{P}_l$ fulfilling %such that for each $x\in G_l$
    \begin{equation}\label{KI}
    \left|\frac{\partial u}{\partial x_i}\right|^2\chi_{G_l}\leq C\sum_{P\in \mathcal{P}_l}\fint_P\left|\frac{\partial^2 u}{\partial x_i^2}\right|dy\fint_P|u|dy \chi_P
    \end{equation}
    and 
    $$
    \sum_{P\in \mathcal{P}_l} \chi_P \leq 5\chi_{\er^n}.
    $$    
\end{Lemma}
\begin{proof}
    Let $u\in\mathcal{C}^2\cap(L^1+L^\infty)(\er^n)$ and $i=1$. We need to overcome rather technical difficulties caused by possible non-compactness of the support of $u$. The proof is based on the proof of Theorem \ref{crux}; thus, we keep the notation from there.

Observe that it follows from the Fubini theorem and the definition of the space $(L^1+L^{\infty})(\er)$ that for a.e. $\bar x\in \er^{n-1}$, function $u(\cdot, \bar x)$ is in $(L^1+L^{\infty})(\er)$. 
Define 
\[
G:=\er \times \{\bar x:u(\cdot,\bar x) \in (L^1+L^\infty)(\er)\}. 
\]
Thus, $\er^n\backslash G$ is of zero measure. 

Considering $Q(0,l):=[-l,l]^n$, we define 
\[
S_k^l=\{(t,\bar x)\in E_k\cap G:(z_{\bar x}(t),y_{\bar x}(t))\subset Q(0,l)\}.  
\] 
For each $l\in \mathbb{N}$ and each $k\in \mathbb{Z}$, there is $\delta_k^l>0$ such that 
for all $x,y\in Q(0,l+1)$, $|x-y|\leq \delta_k^l$ implies 
$$
   \max\left\{ \left|u(x)-u(y)\right|, \left|\frac{\partial u}{\partial x_1}(x)-\frac{\partial u}{\partial x_1}(y)\right|,\left|\frac{\partial^2 u}{\partial x_1^2}(x)-\frac{\partial^2 u}{\partial x_1^2}(y)\right| \right\}\leq \min\left\{2^{\frac{k}{2}-2},2^{k-4}, \frac{2^{k-4}}{M_l}\right\},
$$
where 
$$M_l=\max\left\{\left\|u\right\|_{L^\infty(Q(0,l+1))},\left\|\frac{\partial u}{\partial x_1}\right\|_{L^\infty(Q(0,l+1))},\left\|\frac{\partial^2 u}{\partial x^2_1}\right\|_{L^\infty(Q(0,l+1))}\right\}.$$
Define 
\begin{equation}\label{Rkl}
R_k^l=\bigcup_{(t,\bar x)\in S_k^l}(z_{\bar x}(t),y_{\bar x}(t))\times B(\bar x,\delta^l_k).
\end{equation}
Explaining exactly as in the proof of Theorem \ref{crux}, we conclude that there is universal constant $C>0$, independent on $k$ and $l$, such that
$$
\left(\frac{\partial u}{\partial x_1}(t,\bar x)\right)^2\leq 
C \fint_{R_k^l}\left|\frac{\partial^2 u}{\partial x^2_1}(x)\right|dx \fint_{R_k^l}|u(x)|dx,
$$ 
holds for each $(t,\bar x)\in R_k^l$ (compare with explanation after \eqref{remmm}).
Finally, we define 
\[
G_l=\bigcup_{k\in \mathbb{Z}}R_k^l{\rm \ and\ } \mathcal{P}_l^+=\bigcup_{k}\{R_k^l\}.
\]
From the proof of Lemma \ref{sparse1dim}(check the argumentation on finitness of interval defined in \eqref{Intervaly}) it follows that 
\[
\bigcup_{l\in \mathbb{N}}S^k_l=G\cap E_k.
\]
In consequence, since by respective definitions $S_k^l\subset R_k^l$, we conclude 
\[
G\cap E_k\subset \bigcup_{l\in \mathbb{N}}R^k_l.
\]
Moreover, for each $k\in \mathbb{Z}$ and each $l\in \mathbb{N}$, $S_k^l\subset S_k^{l+1}$. Notice that unlike $(S_k^{l})_l$, the sequence $(R_k^{l})_l$ need not to be increasing, because $\delta_k^l$ depends on $l$.  Anyhow, monotonicity of $(S_k^{l})_l$ forces that $\chi_{G_l}\to \chi_{E_k}$ a.e. with $l\to \infty$. All together it means that $\chi_{G_l}\to \chi_{\{\frac{\partial u}{\partial x_1}>0\}}$ a.e. with $l\to \infty$.  

\end{proof}

%%%%%%%%%%%%%%%%%%%%%%%%%%%%%%%%%%%%%%%%%%%%%%%%%%%%%55
%%%%%%%%%%%%%%%%%%%%%%%%%%%%%%%%%%%%%%%%%%%%%%%%%%

\section{Gagliardo--Nirenberg inequality}\label{GNsection}
In the following, we use Banach function spaces in the sense of  \cite[Definition 1.1, pg. 2]{BS} over $\er^n$ with the Lebesgue measure. In particular, all such spaces have the Fatou property, i.e. for each $u\in L^0$ and each sequence $(u_n)\subset X$ satisfying $u_n(x)\to u(x)$ for a.e. $x\in\er^n$ one has
$$
\|u\|_X\le \liminf_{n\to\infty}\|u_n\|_X,
$$
(see \cite[Lemma 1.5, pg. 4]{BS}).
Moreover, all Banach function spaces under consideration are additionally rearrangement invariant (r.i. for short) according to \cite[Definition 4.1, pg. 59]{BS}.

We also need the Calder\'on--Lozanovskii construction (see, for example, \cite{Lo72, Re88}).

\begin{Definition}[Calder\'on--Lozanovskii construction]
Let $X, Y$ be a couple of Banach function spaces defined on $\er^n$. For $\theta\in(0,1)$ the Calder\'on--Lozanovskii space $X^\theta Y^{1-\theta}$ is defined as
$$
X^\theta Y^{1-\theta}:=\{h\in L^0: |h|\leq f^\theta g^{1-\theta} { \rm\ for\ some\ } f\in X, g\in Y\}
$$
and equipped with the norm
$$
\|h\|_{X^\theta Y^{1-\theta}}:=\inf\{\|f\|_X^{\theta}\|g\|_Y^{1-\theta}: |h|\leq f^\theta g^{1-\theta}\}.
$$
\end{Definition}
A simple consequence of the above definition, which will be later used without mentioning, is that
$$
\|fg\|_{X^\theta Y^{1-\theta}}\leq \|f\|_X^\theta \|g\|_Y^{1-\theta}.
$$
For arbitrary $\theta\in(0,1)$ and r.i. Banach function spaces $X,Y$, also  $X^\theta Y^{1-\theta}$ is an r.i. Banach function space.

The following two lemmas are straightforward and known; we formulate and prove them for the sake of convenience.

\begin{Lemma}\label{lemCM}
    Let $X$ be r.i. Banach function space on $\mathbb{R}^n$ and let $\mathcal{P}$ be countable family of finite measure sets in $\er^n$, such that 
    \begin{equation}\label{Kprekriv}
      \sum_{P\in\mathcal{P}}\chi_P\leq K\chi_{\mathbb{R}^n}, 
    \end{equation}
    for some constant $K$. Then it holds
    \[
    \|T_{\mathcal{P}}\|_{X\to X}\leq K.
    \]
\end{Lemma}
\begin{proof}
Recall that each r.i. Banach function space on $\er^n$ is the exact interpolation space for the couple $(L^1(\mathbb{R}^n),L^{\infty}(\mathbb{R}^n)) $ by the Calder\'on--Mitjagin theorem \cite[pg. 105, Theorem 2.2]{BS}. Thus, it is enough to notice that 
  \[
    \|T_{\mathcal{P}}\|_{L^1\to L^1}\leq K {\rm\ and\ } \|T_{\mathcal{P}}\|_{L^{\infty}\to L^{\infty}}\leq K,
  \]
  which is, however, straightforward.
\end{proof}

The next may be regarded as the Young inequality for r.i. Banach function spaces. 

\begin{Lemma}\label{Young}
Let $X$ be a r.i. Banach function space on $\er^n$. Then for each $f\in X$, $\phi\in L^1(\er^n)$ there holds 
\[
\|f*\phi\|_X\leq \|f\|_X\|\phi\|_{L^1}.
\]
\end{Lemma}
\begin{proof}
    It follows directly from the Young inequality and the Calder\'on--Mitjagin theorem. Let us explain; first, consider operator $T_{\phi}:g\mapsto g*\phi$. Then Young inequality ensures that it is bounded on $L^1$ and on $L^{\infty}$ with both norms dominated (in fact, equal to) by $\|\phi\|_{L^1}$. By the Calder\'on--Mitjagin theorem  \cite[pg. 105, Theorem 2.2]{BS} we know that $X$ is an exact interpolation space for the couple $(L^1,L^{\infty})$, which means that it also holds
    \[
    \|T_{\phi}\|_{X\to X}\leq \|\phi\|_{L^1}. 
    \]
\end{proof}

Let now to the end of the paper $(\phi_{l})$ be the standard mollifying kernel, i.e. $\phi_l(x)=\phi(x/l)$ for some $0\leq \phi\in C^{\infty}_0$ such that $\|\phi\|_{L^1}=1$ and  $\supp \phi \subset B(0,1)$.

The main result will be made by induction based on the following particular case.

\begin{Lemma}\label{mainspec}
Let $X, Y$ be r.i. Banach function spaces over $\mathbb{R}^n$ and let $Z:=X^{1/2}Y^{1/2}$ be  the Calder\'on--Lozanovskii space. Then there are two positive constants $C_{a}$ and $C_{b}$ such that  for each $u\in W^{2,1}_{\loc}(\er^n)$ and each $1\leq i\leq n$ the following inequalities hold
\begin{equation}\label{GNrifin1a}
\left\|\frac{\partial u}{\partial x_i}\right\|_{Z}^2\leq C_{a}\left\|\frac{\partial^2 u}{\partial x_i^2} \right\|_X\|u\|_Y
\end{equation}
and
\begin{equation}\label{GNrifin1b}
\|\nabla u\|_{Z}^2\leq C_{b}\|\nabla^2 u\|_X\|u\|_Y.
\end{equation}
\end{Lemma}

\begin{Remark}
Notice that if $u\in \mathcal{C}_c^2(\mathbb{R}^n)$, then by  Corollary \ref{corcrux} and Lemma \ref{lemCM} we have directly 
$$
\begin{aligned}
\|\nabla u\|_{X^{\frac{1}{2}}Y^{\frac{1}{2}}}&\leq C\|\left(T_{\mathcal{P}}|\nabla^2 u|\right)^{\frac{1}{2}}\left(T_{\mathcal{P}}|u|\right)^{\frac{1}{2}}\|_{X^{\frac{1}{2}}Y^{\frac{1}{2}}}\\     
&\leq C\|T_{\mathcal{P}}|\nabla^2 u|\|_X^{\frac{1}{2}}\|T_{\mathcal{P}}|u|\|_Y^{\frac{1}{2}}\\
&\leq C\|\nabla^2 u\|_X^{\frac{1}{2}}\|u\|_Y^{\frac{1}{2}},
\end{aligned}
$$
as required, while $C$ is constant, differing in steps but depending only on dimension. This argument would be enough to prove Lemma \ref{mainspec}  also if spaces $X, Y$ are separable. However, we do not assume the separability of considered spaces. In consequence, the Fatou property has to replace separability, but then the proof complicates slightly, and we need Lemma \ref{Almpoint} rather than Theorem \ref{crux}. 
\end{Remark}

\begin{proof}[Proof of Lemma \ref{mainspec}]
We start with the proof of \eqref{GNrifin1a}. Let $u\in W^2(L^1+L^{\infty})(\er^n)$ and $\phi_l$ be a mollifier  defined as above. We put  
$$
u_l:=u*\phi_l,\ l\in \mathbb{N}. 
$$
Then, each $u_l$ satisfies assumptions of Lemma \ref{Almpoint}. Applying it for arbitrary fixed $1\leq i\leq n$ and  $l$  we get family $(G_k)$ satisfying 
   \[
    \left|\frac{\partial u_l}{\partial x_i} \right|^2\chi_{G_k}\leq C \sum_{P\in \mathcal{P}_k} \fint_{P}|u_l(z)|dz \fint_{P}\left|\frac{\partial^2 u_l}{\partial x_i^2}(z)\right|dz \chi_P,
  \]
and 
 \[
        \sum_{P\in \mathcal{P}_k}\chi_P\leq 5,
 \] where $C$ is constant depending only on dimension.
Thus, by Lemma \ref{lemCM}  and the Fatou property
\[
\left\|\frac{\partial u_l}{\partial x_i}\right\|^2_{Z}\leq \liminf_{k\to \infty} \left\|\frac{\partial u_l}{\partial x_i}\chi_{G_k}\right\|^2_{Z}\leq 
C \left\|\left(T_{\mathcal{P}_k}\left|\frac{\partial^2 u_l}{\partial x_i^2}\right|\ T_{\mathcal{P}_k}\left|u\right|\right)^{\frac{1}{2}}\right\|_Z
\]
\[
\leq C \left\|\frac{\partial^2 u_l}{\partial x_i^2}\right\|_X^{\frac{1}{2}}\left\|u_l\right\|_Y^{\frac{1}{2}}.
\]

Then, by Lemma \ref{Young}  
\[
\left\|\frac{\partial^2 u_l}{\partial x_i^2}\right\|_X^{\frac{1}{2}}\left\|u_l\right\|_Y^{\frac{1}{2}}\leq 
\left\|\phi_l\right\|_{L^1}\left\|\frac{\partial^2 u}{\partial x_i^2}\right\|_X^{\frac{1}{2}}\|u\|_Y^{\frac{1}{2}}  =
\left\|\frac{\partial^2 u}{\partial x_i^2}\right\|_X^{\frac{1}{2}}\|u\|_Y^{\frac{1}{2}}.
\]
Using once again the Fatou property, we conclude that
\[
\left\|\frac{\partial u}{\partial x_i}\right\|_{Z}\leq \liminf_{l\to \infty} \left\|\frac{\partial u_l}{\partial x_i}\right\|_{Z}\leq C\left\|\frac{\partial^2 u}{\partial x_i^2}\right\|_X^{\frac{1}{2}}\left\|u\right\|_Y^{\frac{1}{2}}.
\]

Having \eqref{GNrifin1a}, the inequality \eqref{GNrifin1b} is its easy consequence with appropriately enlarged constant according to \eqref{nabla}. 
\end{proof}

\begin{proof}[Proof of Theorem \ref{MT}] We will prove only \eqref{GNri}, since the proof of \eqref{pureGNri} goes essentially the same lines. 

Let us denote by $GN(j,k)$ the statement that \eqref{GNri} holds for some $1\leq j<k$ and for all r.i. Banach function spaces  $X, Y$. By Lemma \ref{mainspec}, we know that $GN(1,2)$ is true. We will proceed by induction according to the following scheme:
\begin{itemize}
\item[(A)] $GN(1,k-1) \wedge GN(k-2,k-1) \Rightarrow GN(1,k)$
\item[(B)] (in case $1<j<k-1$)  $GN(j-1,k-1) \wedge GN(1,k) \Rightarrow GN(j,k)$
\end{itemize}

We start with the proof of (A)-step. Let $X$ and $Y$ be r.i. Banach function spaces. By $GN(k-2,k-1)$ applied to $\nabla u$ and respective spaces we have 
\[
\|\nabla^{k-1} u\|_{X^{\frac{k-1}{k}}Y^{\frac{1}{k}}} \leq C \|\nabla^k u\|^{\frac{k-2}{k-1}}_X 
\|\nabla u\|^{\frac{1}{k-1}}_{X^{\frac{1}{k}}Y^{\frac{k-1}{k}}},
\]
since 
\[
X^{\frac{k-1}{k}}Y^{\frac{1}{k}}=X^{\frac{k-2}{k-1}}\left(X^{\frac{1}{k}}Y^{\frac{k-1}{k}}\right)^{\frac{1}{k-1}}.
\]
On the other hand, $GN(1,k-1)$ gives 
\[
\|\nabla u\|_{X^{\frac{1}{k}}Y^{\frac{k-1}{k}}} \leq C \|\nabla^{k-1} u\|^{\frac{1}{k-1}}_{X^{\frac{k-1}{k}}Y^{\frac{1}{k}}} 
\| u\|^{\frac{k-2}{k-1}}_{Y}.
\]
Thus we have
\[
\|\nabla u\|_{X^{\frac{1}{k}}Y^{\frac{k-1}{k}}} \leq C \left(\|\nabla^k u\|^{\frac{k-2}{k-1}}_X 
\|\nabla u\|^{\frac{1}{k-1}}_{X^{\frac{1}{k}}Y^{\frac{k-1}{k}}} \right)^{\frac{1}{k-1}}
\| u\|^{\frac{k-2}{k-1}}_{Y},
\]
which after evident cancellations gives $GN(1,k)$ for $X$ and $Y$, i.e. 
\[
\|\nabla u\|_{X^{\frac{1}{k}}Y^{\frac{k-1}{k}}}\leq C \|\nabla^k u\|^{\frac{1}{k}}_X\| u\|^{\frac{k-1}{k}}_{Y}.
\]

Now we can prove step (B). As before, let $X$ and $Y$ be r.i. Banach function spaces and $1<j<k-1$. Firstly, by $GN(1,k)$
we have 
\[
\|\nabla u\|_{X^{\frac{1}{k}}Y^{\frac{k-1}{k}}}\leq C \|\nabla^k u\|^{\frac{1}{k}}_X\| u\|^{\frac{k-1}{k}}_{Y}.
\]
Then  $GN(j-1,k-1)$ applied to $\nabla u$ gives
\[
\|\nabla^{j} u\|_{X^{\frac{j-1}{k-1}}\left( X^{\frac{1}{k}}Y^{\frac{k-1}{k}}\right)^{\frac{k-j}{k-1}}} \leq C \|\nabla^k u\|^{\frac{j-1}{k-1}}_X 
\|\nabla u\|^{\frac{k-j}{k-1}}_{X^{\frac{1}{k}}Y^{\frac{k-1}{k}}}.
\]
Therefore
\[
\|\nabla^{j} u\|_{X^{\frac{j-1}{k-1}}\left( X^{\frac{1}{k}}Y^{\frac{k-1}{k}}\right)^{\frac{k-j}{k-1}}} \leq C
\|\nabla^k u\|^{\frac{j-1}{k-1}}_X \left( \|\nabla^k u\|^{\frac{1}{k}}_X\| u\|^{\frac{k-1}{k}}_{Y}\right)^{\frac{k-j}{k-1}}
=\|\nabla^k u\|^{\frac{j}{k}}_X\| u\|^{\frac{k-j}{k}}_{Y},
\]
since
\[
X^{\frac{j-1}{k-1}}\left( X^{\frac{1}{k}}Y^{\frac{k-1}{k}}\right)^{\frac{k-j}{k-1}}=X^{\frac{j}{k}}Y^{\frac{k-j}{k}},
\]
and the inductive argument is complete.
\end{proof}

%%%%%%%%%%%%%%%%%%%%%%%%%%%%%%%%%%%%%%%%%%%%%%%%%%%%%%%%%%%%%%%555
%%%%%%%%%%%%%%%%%%%%%%%%%%%%%%%%%%%%%%%%%%%%%%%%%%%%%%%%%%%%%%
%%%%%%%%%%%%%%%%%%%%%%%%%%%%%%%%%%%%%%%%%%%%%%%%%%%%%%%%%%%

\section{Applications to Lorentz and Orlicz spaces}\label{orliczsection}

As we declared, our general considerations also give new results in the case of the most significant examples of r.i. Banach function spaces, i.e. for Lorentz and Orlicz spaces. We use standard definitions, but let us recall them briefly. 

Below $f^*$ stands for the nonincreasing rearrangement of a function $f\in L^0(\er^n)$ (cf. for example \cite{BS}), while 
\[
f^{**}(t)=\frac{1}{t}\int_0^tf^*(s)ds.
\]
Then, for $1\leq P,p\leq \infty$ the {\rm Lorentz space} $L^{P,p}$ is given by the norm
%$$
%\|f\|_{P,p}=\left(\int_0^{\infty}[t^{1/P}f^{**}(t)]^p\frac{dt}%{t}\right)^{1/p}
%$$
%\textcolor{blue}{ Consider the following instead
$$
\|f\|_{L^{P,p}}:=\|t^{\frac{1}{P}-\frac{1}{p}}f^{**}\|_{L^p}.
$$
When $P=\infty$ the space $L^{\infty,p}$ is nontrivial only when also $p=\infty$ and then $L^{\infty,\infty}=L^{\infty}$. Similarly, $L^{1,p}$ is Banach space only when $p=1$ and then $L^{1,1}=L^1$. We consider only the Banach spaces case below. 

Based on Theorem \ref{MT}, we conclude Gagliardo--Nirenberg inequality for Lorentz spaces. 
    
\begin{Corollary}
Let $1\leq j< k$ and $1\leq P,p,Q,q\leq \infty$ (with the standard exceptions). Suppose that 
\[
\frac{j}{P}+\frac{k-j}{Q}=\frac{k}{R}\quad \textup{and}\quad \frac{j}{p}+\frac{k-j}{q}=\frac{k}{r}.
\]

  Then, there are positive constants $C_a, C_b$ such that  the following inequalities hold 
  \begin{equation*}
\left\|\frac{\partial^j u}{\partial x^j_i}\right\|_{L^{R,r}}\leq C_a\left\|\frac{\partial^k u}{\partial x^k_i}\right\|^{\frac{j}{k}}_{L^{P,p}}\|u\|^{1-\frac{j}{k}}_{L^{Q,q}}
\end{equation*}
and
\begin{equation*}%\label{GNriLorentz}
\|\nabla^j u\|_{L^{R,r}}\leq C_b\|\nabla^k u\|^{\frac{j}{k}}_{L^{P,p}}\|u\|^{1-\frac{j}{k}}_{L^{Q,q}}
\end{equation*}
for each $u\in W^{k,1}_{\loc}(\er^n)$ and for each $1\leq i\leq n$.
\end{Corollary}
\begin{proof}
The proof is an immediate consequence of Theorem \ref{MT} once we know the representation formula of the Calder{\'o}n--Lozanovskii construction for Lorentz space (cf. \cite{calderon}). This is, however, well-known and reads as follows
    $$
    (L^{P,p})^{\theta}(L^{Q,q})^{1-\theta}=L^{\theta P+(1-\theta)Q,\theta p+(1-\theta)q},
    $$
    where $1\leq P,p,Q,q\leq \infty$ (with the standard exceptions). The choice $\theta=j/k$ finishes the proof. 
    \end{proof}

Notice that in contrast to \cite[Corollary 2.5]{LeRoSo23} and \cite[Corollary 1.3]{FiFoRoSo2}, the above formula allows one (or two) $L^1$ space on the right, thus is new in these cases.

We also complete the picture of the Gagliardo--Nirenberg inequality for Orlicz spaces. Recall that a continuous, convex, nondecreasing function $\varphi:[0,\infty)\to [0,\infty]$ satisfying $\varphi(0)=0$ is called a {\it Young function}. Such $\varphi$ is called the Orlicz function if we assume additionally that it is bijective. For a given Young function $\varphi$ the Orlicz space $L^{\varphi}$ is defined by
\[
L^{\varphi}=\left\{f\in L^0:\rho_{\varphi}(f):=\int_0^{\infty}\varphi\left(\frac{|f(t)|}{\lambda}\right)dt <\infty {\rm \ for\ some \ }\lambda>0 \right\}
\]
and equipped with the Luxemburg norm
\[
\|f\|_{\varphi}=\inf \left\{\lambda>0 :\rho_{\varphi}(f/\lambda)=\int_0^{\infty}\varphi\left(\frac{|f(t)|}{\lambda}\right)dt \leq 1\right\}.
\]

\begin{Corollary}
Let $1\leq j< k$ and let $\varphi_1,\varphi_2, \psi$ be Orlicz functions such that  
\[
\psi^{-1}= (\varphi_1^{-1})^{\frac{j}{k}}(\varphi_2^{-1})^{1-\frac{j}{k}}.
\]
 Then, there are positive constants $C_a, C_b$ such that  the following inequalities hold 
 $$
\left\|\frac{\partial^j u}{\partial x^j_i}\right\|_{L^{\psi}}\leq C_a\left\|\frac{\partial^k u}{\partial x^k_i}\right\|^{\frac{j}{k}}_{L^{\varphi_1}}\|u\|^{1-\frac{j}{k}}_{L^{\varphi_2}}
$$
and
$$
\|\nabla^j u\|_{L^{\psi}}\leq C_b\|\nabla^k u\|^{\frac{j}{k}}_{L^{\varphi_1}}\|u\|^{1-\frac{j}{k}}_{L^{\varphi_2}}
$$
for each $u\in W^{k,1}_{\loc}(\er^n)$ and for each $1\leq i\leq n$.
\end{Corollary}
\begin{proof}
   As previously, we need only to identify the Calder\'{o}n--Lozanovskii space. From  \cite[p. 179]{Mal89} it follows that 
    $$
    (L^{\varphi_1})^{\theta}(L^{\varphi_2})^{1-\theta}=L^{\psi},
    $$
    i.e. Calderón--Lozanovskii space of two Orlicz spaces is once again Orlicz space defined by the Orlicz function $\psi$ determined by the equation
    $$\psi^{-1}= (\varphi_1^{-1})^{\theta}(\varphi_2^{-1})^{1-\theta}.$$
    Thus, the claim follows when $\theta=j/k$. 
\end{proof}

Once again, the above corollary removes unnecessary assumptions from previously known Gagliardo--Nirenberg inequalities for Orlicz spaces (cf. \cite[Corollary 1.4]{FiFoRoSo2} and \cite{KalPP, Kal}).

Let us finish the discussion on Orlicz spaces with some observations that may be of interest. Namely, our method applies without using the interpolation theory in the Orlicz spaces setting, thus also for Lebesgue spaces. It is because Lemma \ref{lemCM} follows directly from Jensen inequality instead of Calder\'{o}n--Mitjagin theorem. In this version, we get even more, i.e. modular inequality, which is generally stronger than the norm inequality. 

\begin{Remark}
Indeed, let $K$ be the constant controlling the number of overlaps in the family of sets $\mathcal{P}$ and let $\varphi$ be an arbitrary Young function. Let $M$ be the union of all sets from the family $\mathcal{P}$. Then, for $T_{\mathcal{P}}$ defined as in \eqref{Oper} we get 
$$
\begin{aligned}
    \rho_{\varphi}\left(\frac{T_{\mathcal{P}}u}{K}\right)&= 
    \int_{\er^n}\varphi\left(\sum_{P\in\mathcal{P}}\frac{\chi_{P}(x)}{K}\fint_P|u(y)|dy\right)dx\\
    &\leq \int_{M}\varphi\left(\sum_{P\in\mathcal{P}}\frac{\chi_{P}(x)}{\sum_{A\in \mathcal{P}}\chi_A(x)}\fint_P|u(y)|dy\right)dx\\
    &\leq \int_{M}\sum_{P\in\mathcal{P}}\frac{\chi_{P}(x)}{\sum_{A\in \mathcal{P}}\chi_A(x)}\varphi\left(\fint_P|u(y)|dy\right)dx\\  &\leq\int_{M}\sum_{P\in\mathcal{P}}\frac{\chi_{P}(x)}{\sum_{A\in \mathcal{P}}\chi_A(x)}\fint_P\varphi(|u(y)|)dy dx\\
    &\leq \int_{M}\varphi(|u(x)|)dx\\
    &\leq \rho_\varphi(u).
\end{aligned}
$$
In particular, it implies the norm inequality, i.e. 
$$
\|T_\mathcal{P} u\|_{L^\varphi}\leq K\|u\|_{L^\varphi}.
$$
\end{Remark}

%%%%%%%%%%%%%%%%%%%%%%%%%%%%%%%%%%%%%%%%%%%%%%%%%%%%%%%%%%%%%%%%%%%%%%%%%%%%%%%%%%%%%%%%%%%%%%%%%%%%%%%%%%%%%%%%%%%%%%%%%%%%%%%%%%%%%%%%%%%%%%%%%%%%%%%%%%%%%%%%%%%%%%%%%%%%%%%%

\section{Summary and  comments}\label{questsection}
%Let us summarise and comment on the advantages and disadvantages of the method applied. 

\begin{Remark}
In the beginning, we wish to explain an unusual form of the Gagliardo--Nirenberg inequality \eqref{pureGNri}, i.e. 
\begin{equation}\label{pureGNrisummary}
\left\|\frac{\partial^j u}{\partial x^j_i}\right\|_{Z}\leq C\left\|\frac{\partial^k u}{\partial x^k_i}\\\right\|^{\frac{j}{k}}_X\|u\|^{1-\frac{j}{k}}_Y,
\end{equation}
and why we decided to distinguish it from \eqref{GNri}. Evidently, \eqref{pureGNrisummary} implies \eqref{GNri}. On the other hand, observe that for Lebesgue spaces, one-dimensional inequality \eqref{GNri} together with the Fubini theorem and the H\"older inequality imply \eqref{pureGNrisummary}. In fact, we have %by one-dimensional \eqref{GNri}

$$%\label{pureGNriLP}
\begin{aligned}
\left\|\frac{\partial^j u}{\partial x^j_1}\right\|_{L^p}&= \left(\int_{\er^{n-1}}\left(\int_{\er}\left|\frac{\partial^j u}{\partial x^j_1}\right|^p dx_1\right)d\bar x\right)^{\frac{1}{p}}\\
&\leq C \left(\int_{\er^{n-1}}\left(\int_{\er}\left|\frac{\partial^k u}{\partial x^k_1}\right|^qdx_1\right)^{\frac{jp}{kq}}\left(\int_{\er}\left|u\right|^rdx_1\right)^{\frac{(k-j)p}{kr}}d\bar x\right)^{\frac{1}{p}}\\
&\leq C \left(\int_{\er^n}\left|\frac{\partial^k u}{\partial x^k_1}\right|^qdx\right)^{\frac{j}{kq}}\left(\int_{\er^n}\left|u\right|^rdx\right)^{\frac{(k-j)}{kr}}d\bar x\\
&=C \left\|\frac{\partial^k u}{\partial x^k_1} \right\|_{L^q}^{\frac{j}{k}}\left\|u \right\|_{L^r}^{1-\frac{j}{k}},
\end{aligned}
$$
where $\frac{1}{p}=\frac{j}{kq}+\frac{k-j}{kr}$.
The same argument, however, cannot be applied to another r.i. Banach function spaces because there is no analogue of the Fubini theorem, even for Orlicz or Lorentz spaces (cf. \cite{BBS02} and references therein). It means there is no easy argument why \eqref{GNri} should imply \eqref{pureGNri} for general r.i. Banach function spaces. Thus, \eqref{pureGNri} seems stronger than \eqref{GNri} in the setting of r.i. Banach function spaces.

The situation is similar for pointwise inequalities. Namely, it does not follow from the proof of 
Ka{\l}amajska--Mazya--Schaposhnikova inequality, that $\nabla$ operator can be replaced therein by partial derivative $\frac{\partial u}{\partial x_i}$, for fixed $i$. Thus, also \eqref{BodOdhad} seems to have no "maximal operator" counterpart, up to our knowledge. 

\end{Remark}

\begin{Remark}
 It is also worth emphasizing that in the statement of Lemma \ref{mainspec} (analogously for Theorem \ref{MT}), the full second gradient on the left is superfluous. Namely, there holds seemingly stronger inequality without mixed derivatives on the right
\begin{equation}\label{GNrifin1comm}
\|\nabla u\|_{Z}^2\leq C \left\|\left(\frac{\partial^2u}{\partial x_i^{2}}\right)_{i=1,...,n}\right\|_X\|u\|_Y.
\end{equation}
Of course, by the classical tools of harmonic analysis, we know that norms of mixed derivatives are dominated by the norm of all pure derivatives, but only in r.i. Banach function spaces with nontrivial upper Boyd index. At the same time, the Ornstein "noninequality" says that the same is impossible in $L^1$ \cite{Orn62, KSW17}. Thus, \eqref{GNrifin1comm} is perhaps stronger than \eqref{GNri} in the case when $X$ has the upper Boyd index equal to one. Notice that for $X=L^1$, it may also be derived from the original proof of Gagliardo and Nirenberg (cf. \cite{FIFOROSO}).   
\end{Remark}

\begin{Remark}
We are also obliged to explain why we use the term sparse domination, although our sparse families are somehow different from those in the classical sparse domination theory. In fact, for example, in \cite{Hy12, Le13a, Le13b, Le16, LeOm20, LeLo22, LeRoSo23},   a family $\mathcal{S}$ of (often dyadic) cubes is said to be sparse if there is some universal constant $0<\eta<1$ such that
\begin{enumerate}[(a)]
    \item for each $Q\in \mathcal{S}$ there is measurable $E(Q)\subset Q$ such that $|E(Q)|>\eta |Q|$, 
    \item the family $(E(Q))_{Q\in \mathcal{S}}$ is disjoint.
\end{enumerate}
On the other hand, we define a sparse family as a family $\mathcal{P}$ of measurable sets satisfying 
\begin{equation}\label{our}
    \sum_{P\in \mathcal{P}}\chi_P\leq K,
\end{equation}
for some constant $K>0$. 
Evidently, family satisfying (a) and (b) need not satisfy \eqref{our} (take $\{(0,2^{-n})\}_{n\in \mathbb{N}}$ for example). Vice versa, in \eqref{our}, we do not even put any restriction on the shape of sets from $\mathcal{P}$. Nevertheless, the principle is to control the overlapping of sets from $\mathcal{P}$ ($\mathcal{S}$, respectively), so the idea remains the same. 
 It seems that differences follow from different purposes. The classical sparseness is best fitted to the weighted inequalities (for Muckenhoupt weights), where it is usually applied, while \eqref{our} ensures uniform boundedness of all $T_{\mathcal{S}}$ in each r.i. Banach function space. From this point of view, the essence of the method is the same, i.e. having an operator $T$ (usually convolution type operator), we look for sparse family $\mathcal{S}$, depending on $f$ and satisfying $Tf\leq C T_{\mathcal{S}}f$, assuming at the same time that the sparseness is uniform in the sense of constant $\eta$ or $K$, respectively.  

 Notice, finally, that the classical sparse domination has already been applied in the very recent paper \cite{DuDS23} to prove the general Gagliardo--Nirenberg inequality for weighted Lebesgue spaces with Muckenhoupt weights. 
 \end{Remark}

\begin{Remark}
   As we mentioned, our method fits r.i. Banach function spaces, and in this shape, it cannot be extended to arbitrary nonsymmetric Banach function space. In fact, uniform boundedness in $X$ of all $T_{\mathcal{P}}$ with $\sum_{P\in \mathcal{P}}\chi_P\leq K$  is equivalent to $X$ being rearrangement invariant (see for example \cite{Kik06}). On the other hand, in the one-dimensional case, we use only sparse families composed of intervals; thus, such $T_{\mathcal{P}}$'s are uniformly bounded when the maximal operator is bounded. Then, however, in the light of Ka{\l}amajska--Mazya--Schaposhnikova inequality, our method gives no new information for norm inequality. 
 
\end{Remark}

\begin{Remark}
    We also wish to point out the simplicity and transparency of the applied method, especially in the one-dimensional case. In fact, the proof of Lemma \ref{sparse1dim} gives a glimpse into the nature of the Gagliardo--Nirenberg inequality. Namely, we see how the inequality \eqref{KLIC} can be factorized into two following ones 
$$
        |u'|\leq C \sum_{P\in \mathcal{P}}\int_P|u''(s)|ds \chi_P
$$
    and 
$$
        |u'|\leq C \sum_{P\in \mathcal{P}}  \frac{1}{|P|^2} \int_P|u(s)|ds \chi_P,
$$
where $P$ is the sparse family from Lemma \ref{sparse1dim}.   
\end{Remark}

\begin{Remark}
    Statements of Theorem \ref{crux} and Corollary \ref{corcrux} hold for smooth functions of compact support, while in a one-dimensional case, the assumption on compact support is not needed. While removing the smoothness assumption may be technically challenging if possible, it seems  Theorem \ref{crux} should hold without assumption on compact support. Thus, the question is if Theorem \ref{crux} is true for each smooth function in $W^2(L^1+L^{\infty})$ (here we mean the function itself and all its derivatives up to the second order belong to $L^1+L^{\infty}$).

\end{Remark}

\renewcommand{\thefootnote}{\fnsymbol{footnote}}

\bibliographystyle{plain}

\begin{thebibliography}{10}

\bibitem{BS}
Colin Bennett and Robert~C. Sharpley.
\newblock {\em Interpolation of operators}.
\newblock Academic press, 1988.

\bibitem{BBS02}
Antonio Boccuto, Alexander~V. Bukhvalov, and Anna~Rita Sambucini.
\newblock Some inequalities in classical spaces with mixed norms.
\newblock {\em Positivity}, 6(4):393--411, 2002.

\bibitem{BrMi}
Ha{\"\i}m Brezis and Petru Mironescu.
\newblock {G}agliardo--{N}irenberg inequalities and non-inequalities: the full
  story.
\newblock {\em Annales de l'Institut Henri Poincar{\'e} C, Analyse non
  lin{\'e}aire}, 35(5):1355--1376, 2018.

\bibitem{calderon}
Alberto~Pedro Calder\'{o}n.
\newblock Intermediate spaces and interpolation, the complex method.
\newblock {\em Studia Math.}, 24:113--190, 1964.

\bibitem{DDN23}
Nguyen~Anh Dao, Jesus~Ildefonso D\'{\i}az, and Quoc-Hung Nguyen.
\newblock Generalized {G}agliardo-{N}irenberg inequalities using {L}orentz
  spaces, {BMO}, {H}\"{o}lder spaces and fractional {S}obolev spaces.
\newblock {\em Nonlinear Anal.}, 173:146--153, 2018.

\bibitem{DuDS23}
Rodrigo Duarte and Jorge Drumond~Silva.
\newblock Weighted {G}agliardo-{N}irenberg interpolation inequalities.
\newblock {\em J. Funct. Anal.}, 285(5):Paper No. 110009, 49, 2023.

\bibitem{FIFOROSO}
Alberto Fiorenza, Maria~Rosaria Formica, Tom{\'a}{\v{s}}~G. Roskovec, and Filip
  Soudsk{\'y}.
\newblock Detailed proof of classical {G}agliardo--{N}irenberg interpolation
  inequality with historical remarks.
\newblock {\em Zeitschrift f{\"u}r Analysis und ihre Anwendungen},
  40(2):217--236, 2021.

\bibitem{FiFoRoSo2}
Alberto Fiorenza, Maria~Rosaria Formica, Tom\'a\v{s} Roskovec, and Filip
  Soudsk\'y.
\newblock Gagliardoâ€“{N}irenberg inequality for rearrangement-invariant
  {B}anach function spaces.
\newblock {\em Rendiconti Lincei - Matematica e Applicazioni}, 30(4):847--864,
  2019.

\bibitem{GNOrlicz-Formica}
Maria~Rosaria Formica.
\newblock A multiplicative embedding inequality in {O}rlicz-{S}obolev spaces.
\newblock {\em JIPAM. J. Inequal. Pure Appl. Math.}, 7(1):Article 33, 6, 2006.

\bibitem{Hy12}
Tuomas~P. Hyt\"{o}nen.
\newblock The sharp weighted bound for general {C}alder\'{o}n-{Z}ygmund
  operators.
\newblock {\em Ann. of Math. (2)}, 175(3):1473--1506, 2012.

\bibitem{Kal}
Agnieszka Ka{\l}amajska.
\newblock Pointwise multiplicative inequalities and {N}irenberg type estimates
  in weighted {S}obolev spaces.
\newblock {\em Studia Math.}, 108(3):275--290, 1994.

\bibitem{KalPP}
Agnieszka Ka{\l}amajska and Katarzyna Pietruska-Pa{\l}uba.
\newblock Interpolation inequalities for derivatives in {O}rlicz spaces.
\newblock {\em Indiana Univ. Math. J.}, 55(6):1767--1789, 2006.

\bibitem{KSW17}
Krystian Kazaniecki, Dmitriy~M. Stolyarov, and Micha{\l} Wojciechowski.
\newblock Anisotropic {O}rnstein noninequalities.
\newblock {\em Anal. PDE}, 10(2):351--366, 2017.

\bibitem{Kik06}
Masato Kikuchi.
\newblock On some mean oscillation inequalities for martingales.
\newblock {\em Publ. Mat.}, 50(1):167--189, 2006.

\bibitem{KC09}
Tengiz Kopaliani and George Chelidze.
\newblock Gagliardo-{N}irenberg type inequality for variable exponent
  {L}ebesgue spaces.
\newblock {\em J. Math. Anal. Appl.}, 356(1):232--236, 2009.

\bibitem{Lec17}
Michael~T. Lacey.
\newblock An elementary proof of the {$A_2$} bound.
\newblock {\em Israel J. Math.}, 217(1):181--195, 2017.

\bibitem{LecMA17}
Michael~T. Lacey and Dar\'{\i}o Mena~Arias.
\newblock The sparse {T}1 theorem.
\newblock {\em Houston J. Math.}, 43(1):111--127, 2017.

\bibitem{Lady}
Olga~A. Lady\v{z}enskaja.
\newblock Solution ``in the large'' to the boundary-value problem for the
  {N}avier-{S}tokes equations in two space variables.
\newblock {\em Soviet Physics. Dokl.}, 123 (3):1128--1131 (427--429 Dokl. Akad.
  Nauk SSSR), 1958.

\bibitem{Le13b}
Andrei~K. Lerner.
\newblock On an estimate of {C}alder\'{o}n-{Z}ygmund operators by dyadic
  positive operators.
\newblock {\em J. Anal. Math.}, 121:141--161, 2013.

\bibitem{Le13a}
Andrei~K. Lerner.
\newblock A simple proof of the {$A_2$} conjecture.
\newblock {\em Int. Math. Res. Not. IMRN}, 2013(14):3159--3170, 2013.

\bibitem{Le16}
Andrei~K. Lerner.
\newblock On pointwise estimates involving sparse operators.
\newblock {\em New York J. Math.}, 22:341--349, 2016.

\bibitem{LeLo22}
Andrei~K. Lerner, Emiel Lorist, and Sheldy Ombrosi.
\newblock Operator-free sparse domination.
\newblock {\em Forum Math. Sigma}, 10:Paper No. e15, 28, 2022.

\bibitem{LeOm20}
Andrei~K. Lerner and Sheldy Ombrosi.
\newblock Some remarks on the pointwise sparse domination.
\newblock {\em J. Geom. Anal.}, 30(1):1011--1027, 2020.

\bibitem{LeRoSo23}
Karol Le\'{s}nik, Tom\'{a}\v{s} Roskovec, and Filip Soudsk\'{y}.
\newblock Optimal {G}agliardo--{N}irenberg interpolation inequality for
  rearrangement invariant spaces.
\newblock {\em Rev. R. Acad. Cienc. Exactas F\'{\i}s. Nat. Ser. A Mat. RACSAM},
  117(4):Paper No. 161, 2023.

\bibitem{Lo72}
Grigori~Ja. Lozanovski\u{\i}.
\newblock Certain {B}anach lattices. {III}, {IV}.
\newblock {\em Sibirsk. Mat. \v{Z}.}, 13:1304--1313, 1420; ibid. 14 (1973),
  140--155, 237, 1972.

\bibitem{Mal89}
Lech Maligranda.
\newblock {\em Orlicz spaces and interpolation}, volume~5 of {\em
  Semin\'{a}rios de Matem\'{a}tica [Seminars in Mathematics]}.
\newblock Universidade Estadual de Campinas, Departamento de Matem\'{a}tica,
  Campinas, 1989.

\bibitem{MS}
Vladimir Mazya and Tatyana Shaposhnikova.
\newblock On pointwise interpolation inequalities for derivatives.
\newblock {\em Math. Bohem.}, 124(2-3):131--148, 1999.

\bibitem{MRR}
David~S. McCormick, James~C. Robinson, and Jos\'e~L. Rodrigo.
\newblock Generalised {G}agliardo--{N}irenberg inequalities using weak
  {L}ebesgue spaces and {BMO}.
\newblock {\em Milan J. Math.}, 81(2):265--289, 2013.

\bibitem{N}
John Nash.
\newblock Continuity of solutions of parabolic and elliptic equations.
\newblock {\em American Journal of Mathematics}, 80(4):931--954, 1958.

\bibitem{Ni}
Louis Nirenberg.
\newblock On elliptic partial differential equations.
\newblock {\em Annali della Scuola Normale Superiore di Pisa-Scienze Fisiche e
  Matematiche}, 13(2):115--162, 1959.

\bibitem{Orn62}
Donald Ornstein.
\newblock A non-equality for differential operators in the {$L_{1}$} norm.
\newblock {\em Arch. Rational Mech. Anal.}, 11:40--49, 1962.

\bibitem{Re88}
Shlomo Reisner.
\newblock On two theorems of {L}ozanovski\u{\i} concerning intermediate
  {B}anach lattices.
\newblock In {\em Geometric aspects of functional analysis (1986/87)}, volume
  1317 of {\em Lecture Notes in Math.}, pages 67--83. Springer, Berlin, 1988.

\bibitem{WeiWY23}
Wei Wei, Yanqing Wang, and Yulin Ye.
\newblock Gagliardo-{N}irenberg inequalities in {L}orentz type spaces.
\newblock {\em J. Fourier Anal. Appl.}, 29(3):Paper No. 35, 30, 2023.

\end{thebibliography}

\end{document}